\documentclass{amsart}

\usepackage{amssymb,amsmath,latexsym,amsthm}
\usepackage{graphicx}
\usepackage{fontenc}%
\usepackage{marvosym}
\usepackage{paralist}
\usepackage{color}

\newtheorem{theorem}{Theorem}[section]

\newtheorem{lemma}[theorem]{Lemma}
\newtheorem{example}[theorem]{Example}

\newtheorem{corollary}[theorem]{Corollary}
\newtheorem{remark}[theorem]{Remark}

\usepackage{pstricks,pst-text,pst-grad,pst-node,pst-3dplot,pstricks-add,pst-poly}
\definecolor{pink}{rgb}{1, .75, .8}
\psset{arrows=->, labelsep=3pt, mnode=circle}

\definecolor{lgrey}{gray}{.85}

\newcommand{\norm}[1]{\left\|#1\right\|}  

\begin{document}

\title[Proximal Voronoi Regions]{Proximal Vorono\"{i} Regions}

\author[J.F. Peters]{J.F. Peters}
\email{James.Peters3@umanitoba.ca}
\address{Computational Intelligence Laboratory,
University of Manitoba, WPG, MB, R3T 5V6, Canada}
\thanks{The research has been supported by the Scientific and Technological Research
Council of Turkey (T\"{U}B\.{I}TAK) Scientific Human Resources Development
(BIDEB) under grant no: 2221-1059B211402463 and Natural Sciences \&
Engineering Research Council of Canada (NSERC) discovery grant 185986.}

\subjclass[2010]{Primary 65D18; Secondary 54E05, 52C20, 52C22}

\date{}

\dedicatory{Dedicated to the Memory of Som Naimpally}

\begin{abstract}
A main result in this paper is the proof that proximal Vorono\"{i} regions are convex polygons.  In addition, it is proved that every collection of proximal Vorono\"{i} regions has a Leader uniform topology.  
\end{abstract}

\keywords{Convex polygon, proximal, Leader uniform topology, Vorono\"{i} region.}

\maketitle

\section{Introduction}
Klee-Phelps convexity~\cite{Klee1949,Phelps1957} and related results~\cite{PetersSG2014arXiv} are viewed here in terms of Vorono\"{i} regions, named after the Ukrainian mathematician Georgy Vorono\"{i}~\cite{Voronoi1903,Voronoi1907,Voronoi1908}.  A nonempty set $A$ of a space $X$ is a \emph{convex set}, provided $\alpha A + (1-\alpha)A\subset A$ for each $\alpha\in[0,1]$~\cite[\S 1.1, p. 4]{Beer1993}.  A \emph{simple convex set} is a closed half plane (all points on or on one side of a line in $R^2 $).  

\begin{lemma}\label{lem:convexity}~{\rm \cite[\S 2.1, p. 9]{Edelsbrunner2014}} The intersection of convex sets is convex.
\end{lemma}
\begin{proof}
Let $A,B\subset \mathbb{R}^2$ be convex sets and let $K = A\cap B$.  For every pair points $x,y\in K$, the line segment $\overline{xy}$ connecting $x$ and $y$ belongs to $K$, since this property holds for all points in $A$ and $B$.   Hence, $K$ is convex.
\end{proof}

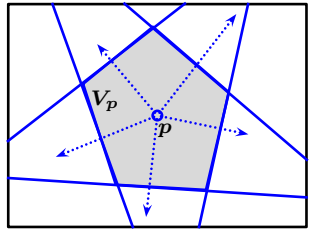
\begin{figure}[!ht]
\begin{center}
 \begin{pspicture}
 (0.0,1.0)(7.5,3.8)
\psset{linewidth=1pt,linecolor=blue}
\psframe[linecolor=black](2.5,0.5)(6.5,3.5)
\rput(4.5,2){\PstPentagon[yunit=1.2,fillstyle=solid,fillcolor=lgrey]}
\pscircle(4.5,2.0){0.08}
\psline[linestyle=dotted,dotsep=0.03]{->}(4.5,2.0)(4.35,0.65)
\psline[linestyle=dotted,dotsep=0.03]{->}(4.5,2.0)(3.70,2.95) 
\psline[linestyle=dotted,dotsep=0.03]{->}(4.5,2.0)(3.15,1.45)
\psline[linestyle=dotted,dotsep=0.03]{->}(4.5,2.0)(5.55,3.35)
\psline[linestyle=dotted,dotsep=0.03]{->}(4.5,2.0)(5.71,1.75)
\psline{-}(5.05,0.5)(5.71,3.5)
\psline{-}(4.05,3.5)(6.50,1.4) 
\psline{-}(2.5,1.17)(6.50,0.91) 
\psline{-}(3.1,3.5)(4.18,0.5) 
\psline{-}(2.5,1.65)(4.88,3.5) 
\rput(4.6,1.8){\footnotesize$\boldsymbol{p}$}
\rput(3.8,2.2){\footnotesize$\boldsymbol{V_p}$}
 \end{pspicture}
\end{center}
\caption[]{$V_p$ = Intersection of closed half-planes}
\label{fig:convexPolygon}
\end{figure}

Let $S\subset \mathbb{R}^2$ be a finite set of $n$ points called sites, $p\in S$.  The set $S$ is called the \emph{generating set}~\cite{Frank2010}. Let $H_{pq}$ be the closed half plane of points at least as close to $p$ as to $q\in S\setminus \{p\}$, defined by
\[
H_{pq } = \left\{x\in R^2:\norm{x - p} \mathop{\leq}\limits_{q\in S} \norm{x - q}\right\}.
\]
A \emph{convex polygon} is the intersection of finitely many half-planes~\cite[\S I.1, p. 2]{Edelsbrunner2001}.  See, for example, Fig.~\ref{fig:convexPolygon}.

\begin{remark} \rm
The Vorono\"{i} region $V_p$ depicted as the intersection of finitely many closed half planes in Fig.~\ref{fig:convexPolygon} is a variation of the representation of a Vorono\"{i} region in the monograph by H. Edelsbrunner~{\rm \cite[\S 2.1, p. 10]{Edelsbrunner2014}}, where each half plane is defined by its outward directed normal vector.  The rays from $p$ and perpendicular to the sides of $V_p$ are comparable to the lines leading from the center of the convex polygon in G.L. Dirichlet's drawing~\cite[\S 3, p. 216]{Dirichlet1850}.
\qquad \textcolor{black}{$\blacksquare$}
\end{remark}

\section{Preliminaries}
Let $S\subset E$, a finite-dimensional normed linear space.  Elements of $S$ are called sites to distinguish them from other points in $E$~\cite[\S 2.2, p. 10]{Edelsbrunner2014}.  Let $p\in S$.  A \emph{Vorono\"{i} region} of $p\in S$ (denoted $V_p$) is defined by
\[
V_p = \left\{x\in E: \norm{x - p} \mathop{\leq}\limits_{\forall q\in S} \norm{x - q} \right\}.
\]

\begin{remark} \rm
A Vorono\"{i} region of a site $p\in S$ contains every point in the plane that is closer to $p$ than to any other site in $S$~\cite[\S 1.1, p. 99]{Frank2010}.  Let $V_p,V_q$ be Vorono\"{i} polygons.  If $V_p\cap V_q$ is a line, ray or line segment, then it is called a \emph{Vorono\"{i} edge}.  If the intersection of three or more Vorono\"{i} regions is a point, that point is called a \emph{Vorono\"{i} vertex}.
\qquad \textcolor{black}{$\blacksquare$}
\end{remark}

\begin{lemma}\label{lem:convexPolygon}
A Vorono\"{i} region of a point is the intersection of closed half planes and each region is a convex polygon.  
\end{lemma}
\begin{proof}
From the definition of a closed half-plane 
\[
H_{pq } = \left\{x\in R^2:\norm{x - p} \mathop{\leq}\limits_{q\in S} \norm{x - q}\right\},
\]
$V_p$ is the intersection of closed half-planes $H_{pq}$, for all $q\in S-\left\{p\right\}$~\cite{Edelsbrunner2001}, forming a polygon.  From Lemma~\ref{lem:convexity}, $V_p$ is a convex.
\end{proof}

A Voronoi diagram of $S$ (denoted by $\mathbb{V}$) is the set of Voronoi regions, one for each site $p\in S$, defined by
\[
\mathbb{V} = \mathop{\bigcup}\limits_{p\in S} V_p.
\]

\begin{example} {\bf Centroids as Sites in an Image Tessellation}.\\
Let $E$ be a segmentation of a digital image and let $S\subset E$ be a set of sites, where each site is the centroid of a segment in $E$.  In a centroidal approach to the Vorono\"{i} tessellation of $E$, a Vorono\"{i} region $V_p$ is defined by the intersection of closed half plains determined by centroid $p\in S$.   The centroidal approach to Voronoi tessellation was introduced by Q. Du, V. Faber, M. Gunzburger~{\rm \cite{Du1999}}. \qquad \textcolor{black}{$\blacksquare$}
\end{example}

\section{Main Results}

Let $V_p,V_z$ be Vorono\"{i} regions of $p,z\in S$, a set of Vorono\"{i} sites in a finite-dimensional normed linear Space $E$ that is topological, $\mbox{cl}A$ the closure of a nonempty set $A$ in $E$.  $V_p,V_z$ are \emph{proximal} (denoted by $V_p\ \delta\ V_z$), provided $\mathbb{P} = \mbox{cl}V_p\cap \mbox{cl}V_z\neq \emptyset$~\cite{Concilio2009}.  The set $\mathbb{P}$ is called a \emph{proximal Vorono\"{i} region}.

\begin{theorem}
Proximal Vorono\"{i} regions are convex polygons.
\end{theorem}
\begin{proof}
Let $\mathbb{P}$ be a proximal Vorono\"{i} region.  By definition, $\mathbb{P}$ is the nonempty intersection of convex sets.  From Lemma~\ref{lem:convexity}, $\mathbb{P}$ is convex.   Consequently, $\mathbb{P}$ is the intersection of finitely many closed half planes.  Hence, from Lemma~\ref{lem:convexPolygon}, $\mathbb{P}$ is a Vorono\"{i} region of a point and is a convex polygon.
\end{proof}

\begin{corollary}
The intersection of proximal Vorono\"{i} regions is either a Vorono\"{i} edge or Vorono\"{i} point.
\end{corollary}

Any two Vorono\"{i} regions intersect at least a vertex and at most along their boundaries.   Together, the set of Vorono\"{i} regions $\mathbb{V}$ cover the entire plane~\cite[\S 2.2, p. 10]{Edelsbrunner2001}.   For a set of sites $S\subset E$, a Vorono\"{i} diagram $\mathbb{D}$ of $S$ is the set of Voronoi regions, one for each site in $S$.  

\begin{corollary}
A Vorono\"{i} diagram $\mathbb{D}$ equals $\mathbb{V}$.
\end{corollary}

The partition of a plane $E$ with a finite set of $n$ sites into $n$ Vorono\"{i} polygons is known as a Dirichlet tessellation, named after G.L. Dirichlet~\cite{Weisstein2014Voronoi} (see~\cite{Dirichlet1850}).   A \emph{cover} (covering) of a space $X$ is a collection $\mathcal{U}$ of subsets of $X$ whose union contains $X$ ({\em i.e.}, $\mathcal{U}\supseteq X$)~\cite[\S 15]{Willard1970},~\cite[\S 7.1]{Naimpally2013}.

\begin{corollary}
A Dirichlet tessellation $\mathbb{D}$ of the Euclidean plane $E$ is a covering of $E$.
\end{corollary}

Recall that the Euclidean space $E = R^2$ is a metric space.  The topology in a metric space results from determining which points are close to each set in the space.  A point $x\in E$ is close to $A\subset E$, provided the Hausdorff distance $d(x,A) = inf\left\{\norm{x-a}: a\in A\right\} = 0$.  Let $X,Y$ be a pair of metric spaces, $f:X\longrightarrow Y$ is a function such that for each $x\in X$, there is a unique $f(x)\in Y$.  A continuous function preserves the closeness (proximity) between points and sets, {\em i.e.}, $f(x)$ is close to $f(B)$ whenever $x$ is close to $B$.  In a proximity space, one set $A$ is near another set $B$, provided $A\ \delta\ B$, {\em i.e.}, the closure of $A$ has at least one element in common with the closure of $B$.  The set $A$ is close to the set $B$, provided the \u{C}ech distance $D(A,B) = inf\left\{\norm{a-b}:a\in A, b\in B \right\} = 0$. In that case, we write $A\ \delta\ B$ ($A$ and $B$ are proximal). A \emph{uniformly continuous mapping} is a function that preserves proximity between sets, {\em i.e.}, $f(A)\ \delta\ f(B)$ whenever $A\ \delta\ B$.  A \emph{Leader uniform topology} is determined by finding those points that are close to each given set in $E$.

\begin{theorem}\label{thm:sites}
Let $S$ be a set of two or more sites, $p\in S, V_p\in \mathbb{D}$ in the Euclidean space $R^2$.  Then
\begin{compactenum}[1$^o$]
\item $V_p$ is near at least one other Vorono\"{i} region in $\mathbb{D}$.
\item Let $p,y$ be sites in $S$.  $\left\{y\right\}\ \delta\ \left\{p\right\}\Rightarrow \left\{y\right\}\ \delta\ V_p$.
\item $V_p$ is close to Vorono\"{i} region $V_y$ if and only if $d(x,V_y) = 0$ for at least one $x\in V_p$.
\item A mapping $f:V_p\longrightarrow V_y$ is uniformly continuous, provided $f(V_p)\ \delta\ f(V_y)$ whenever $V_p\ \delta\ V_y$.
\end{compactenum}
\end{theorem}
\begin{proof}$\mbox{}$\\
1$^o$: Assume $S$ contains at least 2 sites.  Let $p\in S, y\in S\setminus \{y\}$ such that $V_p,V_y$ have at least one closed half plane in common.  Then $V_p\ \delta\ V_y$.\\
2$^o$: If $\left\{y\right\}\ \delta\ \left\{p\right\}$, then $\norm{y-p} = 0$, since $y\in \left\{y\right\}\ \cap\ \left\{p\right\}$.  Consequently, $\left\{y\right\}\ \cap\ \mbox{cl}(V_p)\neq \emptyset$. Hence, $\left\{y\right\}\ \delta\ \mbox{cl}(V_p)$.\\
3$^o$: $V_p\ \delta\ V_y\Leftrightarrow\ \mbox{exists}\ x\in \mbox{cl}(V_p)\ \cap\ \mbox{cl}(V_y)\Leftrightarrow\ d(x,V_y) = 0$.\\
4$^o$: Let $f(V_p)\ \delta\ f(V_y)$ whenever $V_p\ \delta\ V_y$.  Then, by definition, $f:V_p\longrightarrow V_y$ is uniformly continuous.
\end{proof}

\begin{theorem}
Every collection of proximal Vorono\"{i} regions has a Leader uniform topology (application of~{\rm \cite{Leader1959}}).
\end{theorem}
\begin{proof}
Assume $\mathbb{D}$ has more than one Vorono\"{i} region. For each $V_p\in \mathbb{D}$, find all $V_y\in \mathbb{D}$ that are close to $V_p$.  For each  $V_p$, this procedure determines a family of Vorono\"{i} regions that are near $V_p$.   Let $\tau$ be a collection of families of proximal Vorono\"{i} regions.  Let $A,B\in \tau$.  $A\cap B\in \tau$, since either $A\cap B = \emptyset$ or, from Theorem~\ref{thm:sites}.l$^o$, there is at least one Vorono\"{i} region $V_p\in A\cap B$, {\em i.e.}, $V_p\ \delta\ A$ and $V_p\ \delta\ B$.  Hence, $A\cap B\in\tau$.  Similarly, $A\cup B\in \tau$, since $V_p\ \delta\ A$ or $V_p\ \delta\ B$ for each $V_p\in A\cup B$.  Also, $\mathbb{D},\emptyset$ are in $\tau$.   Then, $\tau$ is a Leader uniform topology in $\mathbb{D}$.
\end{proof}


\begin{thebibliography}{99}

\bibitem{Beer1993} G. Beer, \emph{Topologies on Closed and Closed Convex Sets}, Kluwer Academic Pub.,Boston, MA, 1993, MR1269778.

\bibitem{Concilio2009} A. Di Concilio, \emph{Proximity: A powerful tool in extension theory, function spaces, hyper-
spaces, Boolean algebras and point-free geometry}, Amer. Math. Soc. Contemporary Math. {\bf 486} (2009), 89–114,
MR2521943.

\bibitem{Dirichlet1850} G.L. Dirichlet, \emph{\"{U}ber die Reduktion der positiven quadratischen Formen mit drei unbestimmten ganzen Zahlen}, Journal f\"{u}r die reine und angewandte {\bf 40} (1850), 221-239.

\bibitem{Du1999} Q. Du, V. Faber and M. Gunzburger, \emph{Centroidal Vorono\"{i} tessellations: Applications and algorithms}, SIAM Review {\bf 41} (1999), no. 4, 637-676, MR1722997.

\bibitem{Edelsbrunner2001} H. Edelsbrunner, \emph{Geometry and topology of mesh generation}, Cambridge University Press, Cambridge, UK, 2001, 209 pp., MR1833977.

\bibitem{Edelsbrunner2014} H. Edelsbrunner, \emph{A Short Course in Computational Geometry and Topology}, Springer, Berlin, 110 pp.

\bibitem{Frank2010} N.P. Frank and S.M. Hart, \emph{A Dynamical System Using the Vorono\"{i} Tessellation}, The Amer. Math. Monthly {\bf 117} (2010), no. 2, 99–112, MRMR2590195.

\bibitem{Klee1949} V.L. KLee, \emph{A characterization of convex sets}, The Amer. Math. Monthly {\bf 56} (1949), no. 4,
247–249, MR0029519.

\bibitem{Leader1959} S. Leader, \emph{On clusters in proximity spaces}, Fundamenta Mathematicae {\bf 47} (1959),
205–213, MRMR0112120.

\bibitem{Naimpally2013} S.A. Naimpally and J.F. Peters, \emph{Topology with applications. Topological spaces via near and far}, World Scientific, Singapore, 2013, xv + 277pp, MR3075111, Zbl 1295.68010.

\bibitem{PetersSG2014arXiv} J.F. Peters,  M.A. \"{O}zt\"{u}rk and M. U\c{c}kun, \emph{Klee-Phelps convex groupoids}, arXiv {\bf 1411.0934v1 [math.GR]} (2014), 1-4.

\bibitem{Phelps1957} R.R. Phelps, \emph{Convex sets and nearest points}, Proc. Amer. Math. Soc. {\bf 8} (1957), no. 4, 790-797, MR0087897.

\bibitem{Voronoi1903} G. Vorono\"{i}, \emph{Sur un probl\`{e}me du calcul des fonctions
asymptotiques}, J. f\"{u}r die reine und angewandte {\bf 126} (1903), 241-282, JFM 38.0261.01.

\bibitem{Voronoi1907} G. Vorono\"{i}, \emph{Nouvelles applications des param\`{e}tres continus \`{a} la th\'{e}orie des formes quadratiques.  Premier M\'{e}moir}, J. f\"{u}r die reine und angewandte {\bf 133} (1908), 97-178, JFM 38.0261.01.

\bibitem{Voronoi1908} G. Vorono\"{i}, \emph{Nouvelles applications des param\`{e}tres continus \`{a} la th\'{e}orie des formes quadratiques. Deuxi\`{e}me M\'{e}moir}, J. f\"{u}r die reine und angewandte {\bf 134} (1908), 198-287, JFM 39.0274.01.

\bibitem{Weisstein2014Voronoi} E.W. Weisstein, \emph{Voronoi Diagram}, Wolfram MathWorld (2014), $\text{http://mathworld.wolfram.com/VoronoiDiagram.html}$.

\bibitem{Willard1970} S. Willard, \emph{General Topology}, Addison-Wesley Pub. Co., Reading, Mass., 1970, xii + 369 pp., MR0264581.

\end{thebibliography}

\end{document}